\documentclass{amsart}

  \usepackage[utf8]{inputenc}
  \usepackage{amsmath}
  \usepackage{amssymb}
  \usepackage{amsthm}
  \usepackage{mathrsfs}
  \usepackage{bbm}
  \usepackage{graphicx}
  \usepackage{color}

	\newtheoremstyle{slanted}
	{}
	{}
	{\slshape}
	{}
	{\bfseries}
	{.}
	{ }
	{}
	
	\theoremstyle{slanted}
	\newtheorem{theo}{Theorem}[section]

	\newtheorem{lemma}[theo]{Lemma}
	\newtheorem{definition}[theo]{Definition}
	\newtheorem{corollary}[theo]{Corollary}

	\def\egdef{:=}
	\def\Id{\mathop{\mbox{Id}}}
	
	\newcommand{\tend}[3][]{\xrightarrow[#2\to#3]{#1}}

	\newcommand{\EE}{\mathbb{E}}

	\def\ind#1{\mathbbmss{1}_{#1}}
	\newcommand{\ZZ}{\mathbb{Z}}

	\newcommand{\RR}{\mathbb{R}}

	\newcommand{\A}{\mathscr{A}}
	\newcommand{\F}{\mathscr{F}}
	\newcommand{\M}{\mathscr{M}}

\def\u{\bigsqcup}
\def\eps{\varepsilon}
\def\T{(T_t)_{t\in\RR}}

\title{ Around King's Rank-One theorems: Flows and $\ZZ^n$-actions}
\author{Élise Janvresse}
\author{Thierry de la Rue}
\author{Valery Ryzhikov}
\thanks{This work is partially supported by the grant NSh 8508.2010.1. The first draft of the paper was written while the third author was visiting the University of Rouen.}
\address{\'Elise Janvresse, Thierry de la Rue:
Laboratoire de Math\'ematiques Rapha\"el Salem, 
Universit\'e de Rouen, CNRS -- 
Avenue de l'Universit\'e -- 
F76801 Saint \'Etienne du Rouvray.}
\email{Elise.Janvresse@univ-rouen.fr\\Thierry.de-la-Rue@univ-rouen.fr}
\address{Valery Ryzhikov: 
Moscow State University, Faculty of Mechanics and Mathematics,
Leninskie Gory, Moscow, 119991 Russia
}
\email{vryzh@mail.ru}

\begin{document}
\bibliographystyle{amsplain}

\begin{abstract}
 We study the generalizations of Jonathan King's rank-one theorems (Weak-Closure Theorem and rigidity of factors) to the case of rank-one $\RR$-actions (flows) and rank-one $\ZZ^n$-actions. We prove that these results remain valid in the case of rank-one flows. In the case of rank-one $\ZZ^n$ actions, where counterexamples have already been given, we prove partial Weak-Closure Theorem and partial rigidity of factors.
\end{abstract}

\maketitle

\section{Introduction}

Very important examples in ergodic theory have been constructed in the class of rank-one transformations, which is closely connected to the notion of transformations with fast cyclic approximation \cite{KS1966}: If the rate of approximation is sufficiently fast, then the transformation will be inside the rank-one class.
The notion of rank-one transformations has been defined in \cite{Ornstein1972}, where mixing examples have appeared. Later, Daniel Rudolph used them for a machinery of counterexamples \cite{Rudolph1979}.

Jonathan King contributed to the theory of rank-one transformations by several deep and interesting facts. His Weak-Closure-Theorem (WCT)~\cite{King1986} is now a  classical result with applications even out of the range of $\ZZ$-actions (see for example \cite {Tikhonov2006}).
He also proved the minimal-self-joining (MSJ) property for rank-one mixing automorphisms (see \cite{King1988}), the rigidity of non-trivial factors~\cite{King1986}, and the weak closure property for all joinings for flat-roof rank-one transformations~\cite{King2001}.

A natural question is whether the corresponding assertions remain true for flows ($\RR$-actions) and for $\ZZ^n$-actions.  We show that for flows the situation is quite similar:  The joining proof of the Weak-Closure Theorem given in~\cite{Ryzhikov1993} (see also~\cite{Ryzhikov2010}) can be adapted to the situation of a rank-one $\RR$-action (Theorem~\ref{thm:WC}).
We also give in the same spirit a proof of the rigidity of non-trivial factors of rank-one flows (Theorem~\ref{thm:factors}) which, with some simplification, provides a new proof of King's result in the case of $\ZZ$-actions. We prove a flat-roof flow version as well (Theorem~\ref{thm:flat roof}).
Note that a proof of the Weak-Closure Theorem for rank-one flows had already been published in~\cite{Zeitz1993}. Unfortunately it relies on the erroneous assumption that if $\T$ is a rank-one flow, then there exists a real number $t_0$ such that $T_{t_0}$ is a rank-one transformation (see beginning of Section~3.2 in~\cite{Zeitz1993}).

Concerning multidimensional rank-one actions, the situation is quite different. The Weak-Closure Theorem is no more true \cite{DK2002}, and factors may be non-rigid~\cite{DS2009}. Rank-one partially mixing $\ZZ$-actions have MSJ  \cite{KT1991}, however it is proved in \cite{DS2009} that for $\ZZ^2$-actions this is generally not true. We remark that  it was an answer for $\ZZ^2$-action to Jean-Paul Thouvenot's question: Whether a mildly mixing rank-one action possesses MSJ,  though this interesting  problem remains open for $\ZZ$-actions. Regardless these surprising results, there are some partial versions of WCT: Commuting automorphisms can be partially approximated by elements of the action (Corollary~\ref{Cor:partialWCT}), and non-trivial factors must be partially rigid (Corollary~\ref{Cor:partial rigidity}). We present these results as consequences of A. Pavlova's theorem (Theorem~\ref{thm:Pavlova}, see also \cite{Ryzhikov2008}) .

\section{Preliminaries and notations}
\label{Sec:def}

\subsection*{Weak convergence of probability measures}

We are interested in groups of automorphisms of a Lebesgue space $(X,\A,\mu)$, where $\mu$ is a continuous probability measure. The properties of these group actions are independent of the choice of the underlying space $X$, and for practical reasons we will assume that $X=\{0,1\}^\ZZ$, equipped with the product topology and the Borel $\sigma$-algebra. This $\sigma$-algebra is generated by the cylinder sets, that is sets obtained by fixing a finite number of coordinates. On the set $\M_1(X)$ of Borel probability measures on $X$, we will consider the topology of weak convergence, which is characterized by 
$$ \nu_n\tend[w]{n}{\infty} \nu\quad \Longleftrightarrow\quad  \text{for all cylinder set }C,\ \nu_n(C)\tend{n}{\infty} \nu(C), $$
and turns $\M_1(X)$ into a compact metrizable space. 

We will often consider probability measures on $X\times X$, with the same topology of weak convergence. We will use the following observation: If $\nu_n$ and $\nu$ in $\M_1(X\times X)$ have their marginals absolutely continuous with respect to our reference measure $\mu$, with bounded density, then the weak convergence of $\nu_n$ to $\nu$ ensures that for all \emph{measurable} sets $A$ and $B$ in $\A$, $\nu_n(A\times B)\tend{n}{\infty} \nu(A\times B)$.

\subsection*{Self-joinings}

Let $T=(T_g)_{g\in G}$ be an action of the Abelian group $G$ by automorphism of the Lebesgue space $(X,\A,\mu)$. A \emph{self-joining} of $T$ is any probability measure on $X\times X$ with both marginals equal to $\mu$ and invariant by $T\times T=(T_g\times T_g)_{g\in G}$. For any automorphism $S$ commuting with $T$, we will denote by $\Delta_S$ the self-joining concentrated on the graph of $S^{-1}$, defined by
$$ \forall A,B\in\A,\ \Delta_S(A\times B)\egdef \mu(A\cap SB). $$
In particular, for any $g\in G$ we will denote by $\Delta^g$ the self-joining  $\Delta_{T_g}$. In the special case where $S=T_0=\Id$, we will note simply $\Delta$ instead of $\Delta^0$ or $\Delta_{\Id}$.

If $\F$ is a factor (a sub-$\sigma$-algebra invariant under the action $(T_g)$), we denote by $\mu\otimes_{\F}\mu$  the \emph{relatively independent joining above $\F$}, defined by
$$ \mu\otimes_{\F}\mu (A\times B) \egdef \int_X \EE_\mu[\ind{A}|\F] \ \EE_\mu[\ind{B}|\F]\, d\mu. $$
Recall that $\mu\otimes_{\F}\mu$ coincides with $\Delta$ on the $\sigma$-algebra $\F\otimes\F$.

\subsection*{Flows}
A \emph{flow} is a continuous family $\T$ of automorphisms of the Lebesgue space $(X,\A,\mu)$, with $T_t\circ T_s=T_{t+s}$ for all $t,s\in\RR$, and such that $(t,x)\mapsto T_t(x)$ is measurable. We recall that the measurability condition implies that for all measurable set $A$, $\mu(A\vartriangle T_tA)\tend{t}{0} 0$.

\begin{lemma}
\label{lemma:lambda=nu}
 Let $\T$ be an ergodic flow on $(X,\A,\mu)$. Let $Q$ be a dense subgroup of $\RR$, and $\lambda$ be an invariant probability measure for the action of $(T_t)_{t\in Q}$. Assume further that $\lambda\ll\mu$, with $\frac{d\lambda}{d\mu}$ bounded by some constant $C$. Then $\lambda=\mu$. 
\end{lemma}

\begin{proof}
Let $t\in\RR$, and let $(t_n)$ be a sequence in $Q$ converging to $t$. For any measurable set $A$, we have
$$ \lambda \Bigl( T_t A \, \vartriangle\, T_{t_n} A \Bigr) 
\le C \mu \Bigl( T_t A \, \vartriangle\, T_{t_n} A \Bigr) \tend{n}{\infty} 0. $$
Hence $\lambda (T_{t} A)=\lim_n \lambda (T_{t_n} A)=\lambda(A)$. This proves that $\lambda$ is $T_t$-invariant for each $t\in\RR$. Since $\mu$ is ergodic under the action of $\T$, we get $\lambda=\mu$. 
\end{proof}

\section{Rank-one flows}

\begin{definition}
\label{def:rank-one flow}
 A flow $\T$ is of rank one if there exists a sequence $(\xi_j)$ of partitions
of the form
$$\xi_j= \left\{ E_j,\ T_{s_j}E_j, \ T_{s_j}^2 E_j,\ \dots,  T_{s_j}^{h_j-1}E_j, X\setminus \u_{i=0}^{h_j-1} T_{s_j}^i E_j\right\}$$
such that $\xi_j$ converges to the partition into points (that is, for every measurable set $A$ and every $j$, we can find a $\xi_j$-measurable set $A_j$ in such a way that $\mu(A\vartriangle A_j)\tend{j}{\infty}0$),
$s_j/s_{j+1}$ are integers, $s_j\to 0$ and $s_jh_j\to\infty$.
\end{definition}

Several authors have generalized the notion of a rank-one transformation to an $\RR$-action using continuous Rokhlin towers (see \emph{e.g.} \cite{Prikhodko2001}). One can show that the above definition includes all earlier definitions of rank-one flows with continuous Rokhlin towers. The above definition without the requirement that $s_j/s_{j+1}$ be integers was given by the third author in~\cite{Ryzhikov1993}.

\begin{lemma}
\label{lemma:new s_j}
 Let $\T$ be a rank-one flow. Then the sequences $(s_j)$ and $(h_j)$ in the definition can be chosen so that 
$$
s_j^2h_j\tend{j}{\infty}\infty.
$$
\end{lemma}
\begin{proof}
 Let $(s_j)$ and $(h_j)$ be given as in the definition. Recall that $h_js_j\to\infty$. 
For each $j$, let $n_j>j$ be a large enough integer such that $s_js_{n_j}h_{n_j} >j$. 
Define $\ell_j\egdef s_j/s_{n_j}\in\ZZ_+$. We consider the new partition
$$
\tilde \xi_j\egdef 
\left\{ \tilde E_{j}, T_{s_j}\tilde E_{j}, \cdots , T_{s_j}^{\tilde h_j-1}\tilde E_{j}, X\setminus \u_{i=0}^{\tilde h_j-1} T_{s_j}^i \tilde E_{j}\right\}
$$
where
$$
\tilde E_{j}\egdef \u_{i=0}^{\ell_j-1} T_{s_{n_j}}^i E_{n_j} 
$$
and $\tilde h_j\egdef [h_{n_j}/\ell_j]$.
One can easily check that $\tilde \xi_j$ still converges to the partition into points. Moreover we have 
$s_j^2\tilde h_j=s_j^2[h_{n_j}s_{n_j}/s_j]\to\infty$.
\end{proof}

\begin{lemma}[Choice Lemma for flows, abstract setting]
Let $\T$ be an arbitrary flow, and let $\nu$ be an ergodic invariant measure under the action of $\T$.
Let a family of measures $(\nu^k_j)$ satisfy the conditions:
\begin{itemize}
   \item There exist sequences $(d_j)$ and $(s_j)$ of positive numbers with $d_j\tend{j}{\infty}0$, $s_j/s_{j+1}$ is an integer for all $j$, and $s_j\tend{j}{\infty}0$, such that
for all measurable set $A$ and all $k,j$
\begin{equation}
 \label{eq:almost_invariance_by_Tsj}
\left|\nu^k_j(T_{s_j}A) - \nu^k_j(A)\right| < s_j\, d_j ;
\end{equation}
  \item There exists a family of positive numbers $(a_j^k)$ with $\sum_k a^k_j =1$ for all $j$, such that
\begin{equation}
 \label{eq:convergence_to_nu}
\sum_k a^k_j \nu^k_j \tend[w]{j}{\infty} \nu.
\end{equation}
\end{itemize}
Then there is a sequence $(k_j)$ such that
$\nu^{k_j}_j\tend[w]{j}{\infty}\nu$.
\end{lemma}

\begin{proof}
 Given a cylinder set $B$, an integer $j\ge1$ and $\eps >0$,
we consider the sets $K_j$ of all integers $k$ such that
$$ \nu(B) - \nu^k_j (B) >\eps. $$
Suppose that the (sub)sequence $K_j$
satisfies the condition
$$\sum_{k\in K_j} a^k_j \geq a>0.$$
Let $\lambda$ be a limit point for the sequence of measures
$(\sum_{k\in K_j} a^k_j)^{-1}\sum_{k\in K_j} a^k_j \nu^k_j $. Then $\lambda\neq\nu$ since $\lambda(B)\le \nu(B)-\varepsilon$, but by~\eqref{eq:convergence_to_nu}, we have $\lambda\ll\nu$, and $d\lambda/d\nu\le 1/a$.
Moreover, the measure $\lambda$ is invariant by $T_{s_p}$ for all $p$. Indeed, for $j\ge p$, since $s_p/s_j$ is an integer, we get from~\eqref{eq:almost_invariance_by_Tsj} that
$$\left|\nu^k_j(T_{s_p}A) - \nu^k_j(A)\right| < s_p\, d_j\tend{j}{\infty}0.$$
By Lemma~\ref{lemma:lambda=nu}, it follows that $\lambda=\nu$. The contradiction shows that

$$\sum_{k\in K_j} a^k_j \to 0.$$
Thus, for all large enough $j$, most of the measures $\nu^k_j$ satisfy
$$| \nu^k_j (B) - \nu(B) |< \eps.$$
Let $\{B_1, B_2,\dots\}$ be the countable family of all cylinder sets. Using the diagonal method we find a sequence $k_j$ such that for each~$n$
$$| \nu^{k_j}_j (B_n) - \nu(B_n) | \tend{j}{\infty} 0,$$
\emph{i.e.} $\nu^{k_j}_j\tend[w]{j}{\infty}\nu$.
\end{proof}

\subsection*{Columns and fat diagonals in $X\times X$}
Assume that $\T$ is a rank-one flow defined on $X$, with a sequence $(\xi_j)$ of partitions as in Definition~\ref{def:rank-one flow}. 
For all $j$ and $|k|< h_j-1$, we define the sets $C^k_j\in X\times X$, called \emph{columns}:
$$ 
C^k_j\egdef \u_{\genfrac{}{}{0pt}{}{\scriptstyle 0\le r,\ell \le h_j-1}{\scriptstyle r-\ell = k}} T_{s_j}^{r}E_j\times T_{s_j}^\ell E_j .
$$
Given $0<\delta<1$, we consider the set
$$D^\delta_j\egdef\u_{ k=-[\delta h_j]}^{[\delta h_j]} C^k_j.$$
(See Figure~\ref{fig:columns}.)

\begin{figure}[htp]
 \centering
 \input{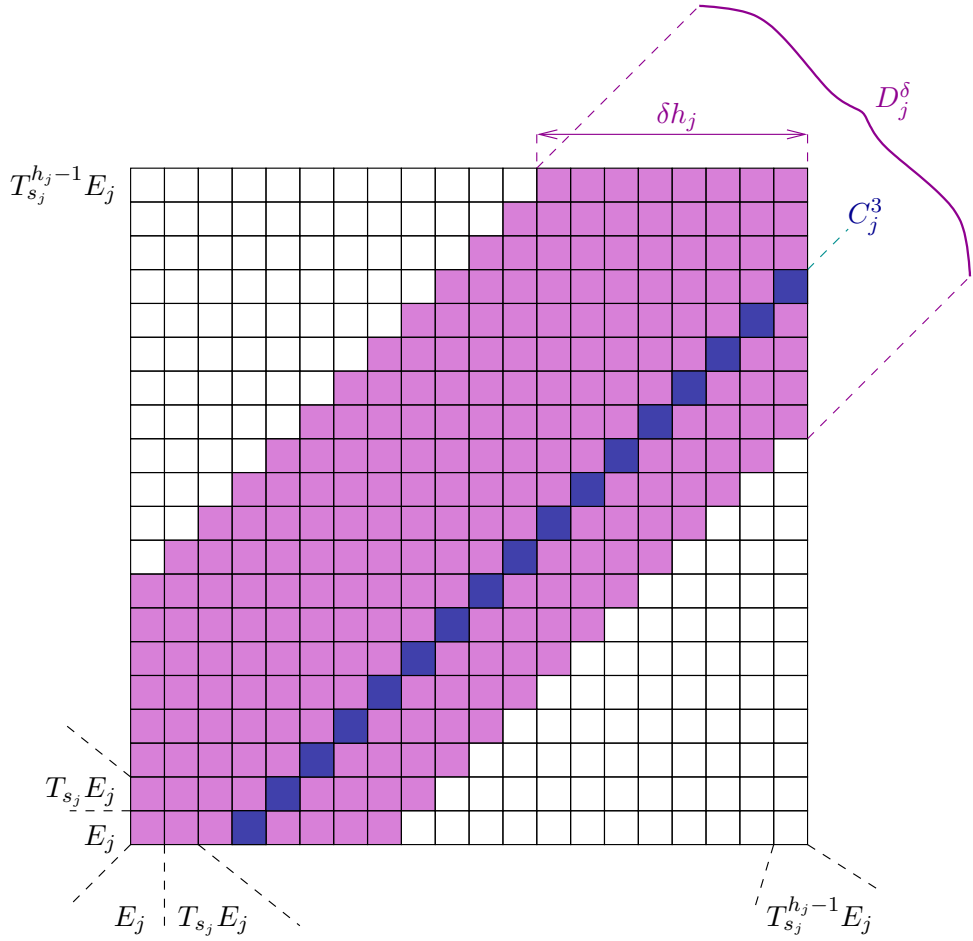}
 \caption{Columns and fat diagonals in $X\times X$}
 \label{fig:columns}
\end{figure}

\section{Approximation theorem}

Recall from Section~\ref{Sec:def} that, given a flow $\T$, $\Delta^t$ stands for the self-joining supported by the graph of $T_{-t}$.
\begin{lemma}
\label{Lemma:approximation}
 Let $\nu$ be an ergodic joining of the rank-one flow $\T$. Let $0<\delta<1$ be such that 
\begin{equation}
 \label{eq:Ddelta_j}
\ell_\delta\egdef\lim_j\nu(D^\delta_j)  > 0.
\end{equation}
Then there exists a sequence $(k_j)$ with $-\delta h_j\le k_j \le \delta h_j$ such that 
$$\Delta^{k_js_j}(  \, \cdot\,  | C_j^{k_j}) \tend[w]{j}{\infty} \nu.$$
\end{lemma}

\begin{proof}
Our strategy is the following: First we prove that the joining $\nu$ can be approximated
by sums of parts of off-diagonal measures, then applying the Choice Lemma
we find a sequence of parts tending to $\nu$.

By definition of $D_j^{\delta}$, we have 
$$ \nu\Bigl(  D^\delta_j\,\vartriangle\,(T_{s_j}\times T_{s_j})D^\delta_j\Bigr) \le \dfrac{C}{h_j}. $$
It follows that for any fixed $p$, the sets $D^\delta_j$ are asymptotically $T_{s_p}\times T_{s_p}$-invariant: Indeed, since $T_{s_p}=T_{s_j}^{s_p/s_j}$ where $s_p/s_j$ is an integer when $j\ge p$, we get
$$ \nu\Bigl(  D^\delta_j\,\vartriangle\,(T_{s_p}\times T_{s_p})D^\delta_j\Bigr) \le \dfrac{s_p}{s_j} \dfrac{C}{h_j} \tend{j}{\infty}0$$
(recall that $s_jh_j\to\infty$). 

Let $\lambda$ be a limit measure of $\nu(\,\cdot\, |\ D^\delta_j)$. 
Then $\lambda$ is $T_{s_p}\times T_{s_p}$-invariant for each $p$, by~\eqref{eq:Ddelta_j}, $\lambda$ is absolutely continuous with respect to $\nu$, and $\frac{d\lambda}{d\nu}\le\frac{1}{\ell_\delta}<\infty$. By Lemma~\ref{lemma:lambda=nu}, it follows that $\lambda=\nu$. Hence we have
\begin{equation}
 \label{eq:convergence of nu}
\nu( \,\cdot\, |\ D^\delta_j)\tend[w]{j}{\infty} \nu.
\end{equation}

We now prove that 
\begin{equation}
 \label{eq:convergence of Delta^k}
\sum_{k=-[\delta h_j]}^{[\delta h_j]} \nu(C^k_j | D^\delta_j) \Delta^{k s_j}(\,\cdot\, | C^k_j)\tend[w]{j}{\infty} \nu.
\end{equation}

 For arbitrary measurable sets $A,B$ we can find
 $\xi_j$-measurable sets $A_j,B_j$ such that
$$\eps_j\egdef\mu(A\vartriangle A_j) +\mu(B\vartriangle B_j) \to 0.$$
We have
$$ \sum_k\nu(C^k_j | D^\delta_j) \Delta^{k s_j}(A\times B| C^k_j)-\nu(A\times B) = M_1+M_2+M_3+M_4, $$
where
\begin{align*}
 & M_1\egdef\sum_k\nu(C^k_j | D^\delta_j) \left(\Delta^{k s_j}(A\times B| C^k_j) -  \Delta^{k s_j}(A_j\times B_j| C^k_j)\right),\\ 
 & M_2\egdef \sum_k\nu(C^k_j | D^\delta_j) \Delta^{k s_j}(A_j\times B_j| C^k_j) - \nu (A_j\times B_j | D^\delta_j),\\
 & M_3\egdef \nu (A_j\times B_j | D^\delta_j) - \nu (A\times B | D^\delta_j), \\
 & M_4\egdef \nu (A\times B | D^\delta_j) - \nu(A\times B). 
\end{align*}
The density of the projections of the measure $\Delta^{k s_j}(\,\cdot\, | C^k_j)$ with respect to $\mu$
is bounded by $(1-\delta)^{-1}$. Hence $ M_1 \le \eps_j / (1-\delta)$. 

Since $A_j,B_j$ are $\xi_j$-measurable, 
$$\nu(A_j\times B_j | C^k_j)= \Delta^{k s_j} (A_j\times B_j | C^k_j),$$
and we get $M_2=0$.

The absolute value of the third term $M_3$ can be bounded above as follows
$$|M_3|\le \nu(D^\delta_j)^{-1} \nu\Bigl((A_j\times B_j)\vartriangle (A\times B)\Bigr) \le \frac{\eps_j}{\nu(D^\delta_j)}\to 0.$$ 

The last term $M_4$ goes to zero as $j\to\infty$ by~\eqref{eq:convergence of nu}, and this ends the proof of~\eqref{eq:convergence of Delta^k}.

To apply the Choice Lemma for the measures 
$\nu^k_j =\Delta^{ks_j}( \,\cdot\,| C_j^k)$ and $a_j^k = \nu(C^k_j | D^\delta_j)$, 
it remains to check the first hypothesis of the lemma. 
By construction of the columns $C_j^k$, we have for any measurable subset $A\in X\times X$ and all $k\in\{-[\delta h_j],\ldots,[\delta h_j]\}$,
\begin{equation}
\label{eq:invariance for Cjk}
 \left|\Delta^{ks_j}( T_{s_j}\times T_{s_j} A  | C_j^k) - \Delta^{ks_j}( A| C_j^k)\right| < \dfrac{C}{h_j}
\end{equation}
where $C$ is a constant. We get the desired result by setting $d_j\egdef \dfrac{C}{s_j h_j}$. 

The Choice Lemma then gives a sequence $(k_j)$ with $-\delta h_j\le k_j \le \delta h_j$ such that $\Delta^{k_js_j}(  \, \cdot\,  | C_j^{k_j}) \tend[w]{j}{\infty} \nu.$
\end{proof}

\begin{theo}
\label{thm:1.2}
 Let  a flow $T=\T$ be of rank-one and $\nu$ be an ergodic self-joining of $\T$.
Then there is a sequence $(k_j)$ such that
$\Delta^{k_js_j} \tend[w]{j}{\infty} \frac{1}{2}\nu +\frac{1}{2}\nu'$
for some self-joining $\nu'$: For all measurable sets $A,B$
$$\mu(A\cap T^{k_j}_{s_j}B) \to \frac{1}{2}\nu(A\times B) +
\frac{1}{2}\nu'(A\times B).$$
\end{theo}

\begin{proof}

For any $1/2<\delta<1$, we have
\begin{equation*}
\lim_{j\to\infty}\nu(D^\delta_j) > 1 -2(1-\delta) =2\delta -1 > 0.
\end{equation*}
Hence we can apply Lemma~\ref{Lemma:approximation} for any $1/2<\delta<1$.
By a diagonal argument, we get the existence of $(k_j)$ and $(\delta_j)\searrow \frac{1}{2}$ with $-\delta_j h_j\le k_j \le \delta_j h_j$ such that 
$$\Delta^{k_js_j}\left( \, \cdot\, | C_j^{k_j}\right) \tend[w]{j}{\infty}\nu. $$
Let us decompose $\Delta^{k_js_j}$ as
$$
\Delta^{k_js_j}=\Delta^{k_js_j}\left( \, \cdot\, | C_j^{k_j}\right) \Delta^{k_js_j}( C_j^{k_j}) + \Delta^{k_js_j}\left( \, \cdot\, | X\times X\setminus C_j^{k_j}\right) \left(1-\Delta^{k_js_j}( C_j^k)\right).
$$
Since $\liminf_{j\to\infty}\Delta^{k_js_j}( C_j^{k_j})\ge 1/2$, we get the existence of some self-joining $\nu'$ such that
$$\Delta^{k_js_j}\tend[w]{j}{\infty} \frac{1}{2}\nu +\frac{1}{2} \nu'.$$
\end{proof}

\begin{corollary}
A mixing rank-one flow has minimal self-joinings of order two.
\end{corollary}

\begin{proof}
 Let $\nu$ be an ergodic self-joining of a mixing rank-one flow $\T$. Let $(k_j)$ be the sequence given by Theorem~\ref{thm:1.2}. If $|k_js_j|\to\infty$, since $T$ is mixing we have 
$$\Delta^{k_js_j}\tend[w]{j}{\infty} \mu\times\mu, $$
hence $\mu\times\mu=\frac{1}{2}\nu +\frac{1}{2} \nu'$ for some self-joining $\nu'$. The ergodicity of $\mu\times\mu$ then implies that $\mu\times\mu=\nu$.  
Otherwise, along some subsequence we have $k_js_j\to s$ for some real number $s$. Then 
$\Delta^s=\frac{1}{2}\nu +\frac{1}{2} \nu'$ for some self-joining $\nu'$, and again the ergodicity of $\Delta^s$ yields $\nu=\Delta^s$. Thus $T$ has minimal self-joinings of order two..
\end{proof}

\section{Weak Closure Theorem for rank-one flows}

\begin{lemma}[Weak Closure Lemma]
\label{lemma:WCLemma}
If the automorphism $S$ commutes with the rank-one flow $\T$, then there exist $1/2\le d \le 1$, a sequence $(k_j)$ of integers and a sequence of measurable sets $(Y_j)$ such that, for all measurable sets $A,B$
$$ \mu(A\cap T^{k_j}_{s_j}B \cap Y_j) \to d\, \mu(A\cap SB),$$
where $Y_j$ has the form
$$Y_j^{d,-} \egdef\u_{0\le i< d h_j}T_{s_j}^i E_j\quad\text{or}\quad
Y_j^{d,+} \egdef\u_{(1-d) h_j < i \le h_j}T_{s_j}^i E_j.$$ 
\end{lemma}

\begin{proof}
 This lemma is a consequence of the proof of Theorem~\ref{thm:1.2}, when the joining $\nu$ is equal to $\Delta_S$. 
Given a sequence $(\delta_j)\searrow \frac{1}{2}$, the proof provides a sequence $(k_j)$ where $-\delta_j h_j\le k_j \le \delta_j h_j$, such that $\Delta^{k_js_j}(  \, \cdot\,  | C_j^{k_j}) \tend[w]{j}{\infty} \Delta_S$, and $\Delta^{k_js_j}(C_j^{k_j})$ converges to some number $d\ge 1/2$. Let $Y_j^{k_j}$ be the projection on the first coordinate of $C_j^{k_j}$, that is 
$$ Y^{k_j}_j=\begin{cases}
              \u_{i=k_j}^{h_j} T_{s_j}^iE_j & \text{ if } k_j\ge 0\\
              \u_{i=0}^{h_j+k_j} T_{s_j}^iE_j & \text{ if } k_j< 0.\\
             \end{cases}
$$
We then have $\Delta^{k_js_j}(  \, \cdot\,  | C_j^{k_j}) = \Delta^{k_js_j}(  \, \cdot\,  | Y_j^{k_j}\times X)$, and $\mu(Y_j^{k_j}) = \Delta^{k_js_j}(C_j^{k_j}) \to d$. 
This yields, for all measurable sets $A,B$,
$$ \mu(A\cap T^{k_j}_{s_j}B \cap Y_j^{k_j}) \to d\, \mu(A\cap SB).$$
If there exist infinitely many $j$'s such that $k_j\ge 0$, then along this subsequence, we have 
$$ \mu\left(Y_j^{k_j}\vartriangle Y_j^{d,+}\right) \tend{j}{\infty} 0, $$
since $(h_j-k_j)/h_j\to d$. A similar result holds along the subsequence of $j$'s such that $k_j<0$, with $Y_j^{d,+}$ replaced by $Y_j^{d,-}$. 
\end{proof}

\begin{theo}[Weak Closure Theorem for rank-one flows]
\label{thm:WC}
If the automorphism $S$ commutes with the rank-one flow $\T$, then there exists a sequence of integers $(k_j)$ such that $\Delta^{k_js_j}\to \Delta_S$: For all measurable sets $A,B$, 
$$ \mu(A\cap T^{k_j}_{s_j}B) \to \mu(A\cap SB).$$
\end{theo}

\begin{proof}
We fix $T$ and consider the set of real numbers $d$ for which the conclusion in the statement of Lemma~\ref{lemma:WCLemma} holds. It is easy to show by a diagonal argument that this set is closed. Hence we consider its maximal element, which we still denote by $d$.
(If $d=1$, the theorem is proved.)

So we start from the following statement: We have a sequence of sets $\{Y_j\}$, of the form given in Lemma~\ref{lemma:WCLemma}, such that for all measurable $A,B$
$$ \mu(A\cap T_{s_j}^{k_j}B \cap Y_j) \to d \mu(A\cap SB).$$
Then a similar statement holds when $Y_j$ is replaced by $SY_j$: Indeed, since $S$ commutes with $T$ and $\mu$ is invariant by $S$, we have
\begin{multline*} 
 \mu(A\cap T^{k_j}_{s_j}B \cap SY_j) =
 \mu(S^{-1}A\cap T^{k_j}_{s_j}S^{-1}B \cap Y_j) \\
\tend{j}{\infty} d \, \mu( S^{-1}A\cap SS^{-1}B) = d \, \mu(A\cap SB).
\end{multline*}
Let $\lambda$ be a limit point for the sequence of probability measures $\{\nu_j\}$ defined on $X\times X$ by
$$
\nu_j(A\times B)\egdef \dfrac{1}{\mu(Y_j\cup SY_j)}\mu\left(A\cap T^{k_j}_{s_j}B \cap (Y_j\cup SY_j)\right).
$$
Then $\lambda\le 2\, \Delta_S$.
Moreover, the measure $\lambda$ is invariant by $T_{s_p}\times T_{s_p}$ for all $p$. Indeed, for $j\ge p$, we have
$$ \mu (T_{s_p}Y_j\vartriangle Y_j) = \mu (T_{s_j}^{s_p/s_j}Y_j\vartriangle Y_j)$$
which is of order $\frac{s_p}{s_j h_j}$, hence vanishes as $j\to\infty$. Since $\Delta_S$ is an ergodic measure for the flow $\{T_t\times T_t\}$, we can apply Lemma~\ref{lemma:lambda=nu}, which gives $\lambda=\Delta_S$. 
We obtain
$$ \mu\left(A\cap T^{k_j}_{s_j}B \cap (Y_j\cup SY_j)\right) \to
 u\, \mu(A\cap SB),$$
where $u \egdef \lim_j \mu(Y_j\cup SY_j)$
(if the limit does not exist, then we consider some subsequence of $\{j\}$).

Our aim is to show that $u=1$, which will end the proof of the theorem. Let us introduce
$$ W_j\egdef \left(\u_{0\le i\le h_j} T_{s_j}^i E_j\right) \setminus Y_j. $$
Assume that $u<1$, then (denoting by $Y^c$ the complementary of $Y\subset X$)
$$ \lim_j \Delta_S(W_j\times W_j) = \lim_j \mu(W_j\cap SW_j) = \lim_j \mu(Y_j^c\cap SY_j^c) = 1-u >0. $$
Let us consider the case where $Y_j$ has the form 
$Y_j^{d,-} =\u_{0\le i< d h_j}T_{s_j}^i E_j$. Then $W_j=\u_{dh_j\le i \le h_j}T_{s_j}^i E_j$, and we define for any $\delta'<1-d$
$$ W_j(\delta')\egdef \u_{(1-\delta')h_j< i \le h_j}T_{s_j}^i E_j\subset W_j. $$
In the same way, if $Y_j$ has the form $Y_j^{d,+} =\u_{(1-d) h_j < i \le h_j}T_{s_j}^i E_j$, we set for $\delta'<1-d$
$$ W_j(\delta')\egdef \u_{0< i <\delta' h_j}T_{s_j}^i E_j\subset W_j. $$
In both cases, note that
$$ \Delta_S\Bigl( (W_j\times W_j) \setminus (W_j(\delta')\times W_j(\delta')) \Bigr) \le 2(1-d-\delta'). $$
Thus, for $\delta'$ close enough to $1-d$, we get 
$$ \limsup_j\Delta_S\Bigl(W_j(\delta')\times W_j(\delta')\Bigr)\ge 1-u -2(1-d-\delta') > 0. $$ 
Since $W_j(\delta')\times W_j(\delta')\subset D_j^{\delta'}$, this ensures that 
$$ \limsup \Delta_S(D_j^{\delta'}) > 0. $$
Lemma~\ref{Lemma:approximation} then provides a sequence $(k'_j)$ with $-\delta' h_j\le k'_j \le \delta' h_j$, such that 
$$ \Delta^{k'_j s_j}(\,\cdot\,|C_j^{k'_j}) \tend[w]{j}{\infty} \Delta_S, $$
and the projections $Y_j^{k'_j}$ of $C_j^{k'_j}$ on the first coordinate satisfy 
$$ \lim_j \mu (Y_j^{k'_j}) \ge 1-\delta'>d, $$
which contradicts the maximality of $d$. Hence $u=1$.
\end{proof}

\section{Rigidity of factors of rank-one flows}

\begin{lemma}
 \label{lemma:rigidity}
Let $\F$ be a non-trivial factor of a rank-one flow $\T$. 
Then there exist $1/2\le d \le 1$, a sequence of integers $(k_j)$ with $|k_js_j|\nrightarrow 0$ and a sequence of measurable sets $(Y_j)$ such that, for all measurable sets $A,B\in\F$
$$ \mu(A\cap T^{k_j}_{s_j}B \cap Y_j) \to d\, \mu(A\cap B),$$
where $Y_j$ has the form
$$Y_j^{d,-} \egdef\u_{0\le i< d h_j}T_{s_j}^i E_j\quad\text{or}\quad
Y_j^{d,+} \egdef\u_{(1-d) h_j < i \le h_j}T_{s_j}^i E_j.$$ 
\end{lemma}

\begin{proof}
We start with the relatively independent joining above the factor $\F$ (see Section~\ref{Sec:def}). Since $\F$ is a non-trivial factor, $\mu\otimes_\F\mu\neq\Delta$, hence we can consider an ergodic component $\nu$ such that $\nu(\{(x,x), x\in X\})=0$. 
Observe however that for any sets $A, B\in\F$, we have $\nu(A\times B)=\mu(A\cap B)$.

We repeat the proof of Lemma~\ref{lemma:WCLemma} with $\nu$ in place of $\Delta_S$. 
This provides sequences $(k_j)$ and $(Y_j)$ and a real number $1/2\le d \le 1$, such that for all measurable sets $A, B$
$$
\mu(A\cap T^{k_j}_{s_j}B \cap Y_j) \to d\, \nu(A\times B).
$$
If we had $k_js_j\to 0$, then the left-hand side would converge to $d\,\mu(A\cap B )$, which would give $\nu(A\times B)=\mu(A\cap B)$ for all $A,B\in\A$, and this would contradict the hypothesis that $\nu$ gives measure 0 to the diagonal. 
\end{proof}

\begin{theo}
\label{thm:factors}
 Let $\F$ be a non-trivial factor of a rank-one flow $\T$. 
Then there exists a sequence of integers $(k_j)$ with $|k_js_j|\to \infty$ such that, for all measurable sets $A,B\in\F$
$$ \mu(A\cap T^{k_j}_{s_j}B ) \to \mu(A\cap B).$$
\end{theo}

\begin{proof}
 Again we fix some ergodic component $\nu$ such that $\nu(\{(x,x), x\in X\})=0$. We consider the maximal number $d$ for which the statement of Lemma~\ref{lemma:rigidity} is true.
We thus have a sequence of sets $\{Y_j\}$, of the form given in Lemma~\ref{lemma:rigidity}, such that 
\begin{equation}
\label{eq:convergence}
\forall A,B\in\F , \quad\frac{1}{\mu(Y_j)} \EE_\mu\left[ \ind{A}\ind{T_{s_j}^{k_j}B}\ind{Y_j}\right] \to \mu(A\cap B).
\end{equation}
In the above equation, one can replace $\ind{Y_j}$ by $\phi_j(x)\egdef \EE_\nu[\ind{Y_j}(x')|x]$: Indeed, since $\nu$ coincides with $\Delta$ on $\F\otimes\F$, we have $\ind{A}(x')=\ind{A}(x)$ and $\ind{T_{s_j}^{k_j}B}(x')=\ind{T_{s_j}^{k_j}B}(x)$ $\nu$-a.s. Hence, 
$$
\EE_\mu\left[ \ind{A}\ind{T_{s_j}^{k_j}B}\ind{Y_j}\right]=\EE_\nu\left[ \ind{A}(x)\ind{T_{s_j}^{k_j}B}(x)\ind{Y_j}(x')\right]
= \EE_\mu\left[ \ind{A}(x)\ind{T_{s_j}^{k_j}B}(x) \phi_j(x)\right].
$$
We note that 
\begin{equation}
 \label{eq:phi}
\EE_\mu\left[ |\phi_j-\phi_j\circ T_{s_j}|\right] \le \mu(Y_j\vartriangle T_{s_j}Y_j) = O\left(\frac1{h_j}\right).
\end{equation}
For any $\eps >0$, let 
$$
U_j^{\eps}\egdef \left\{x : \phi_j(x)>\eps \right\}.
$$
We would like to prove that~\eqref{eq:convergence} remains valid with $\ind{Y_j}$ replaced by $\ind{U_j^{\eps}}$ for $\eps$ small enough. To this end, we need almost-invariance of $U_j^{\eps}$ under $T_{s_j}$, which does not seem to be guaranteed for arbitrary $\eps$. Therefore, we use the following technical argument to find a sequence $(\eps_j)$ for which the desired result holds.

Fix $\eps >0$ small enough so that $\mu(U_j^{\eps})>\mu(Y_j)/2$ for all large $j$. By Lemma~\ref{lemma:new s_j}, we can assume that $s_j^2h_j\to\infty$. 
Let $\delta_j=o(s_j)$ such that $(\delta_jh_j )^{-1}=o(s_j)$.
We divide the interval $[\eps/2, \eps]$ into $\eps/(4\delta_j)$ disjoint subintervals of length $2\delta_j$. 
One of these subintervals, called $I_j$, satisfy 
\begin{equation}
 \label{eq:I_j}
\mu \left(\{x: \phi_j(x)\in I_j\}\right) \le \frac{4\delta_j}{\eps}. 
\end{equation}
Let us call $\eps_j$ the center of the interval $I_j$. Observe that
$$
\mu\left( U_j^{\eps_j} \vartriangle T_{s_j}U_j^{\eps_j} \right) 
\le \mu\left(\{x: |\phi_j(x)-\eps_j|<\delta_j\}\right) + \mu\left(\{x: |\phi_j(x)-\phi_j(T_{s_j}(x))|\ge\delta_j\}\right).
$$
By~\eqref{eq:I_j} and~\eqref{eq:phi}, we get that 
\begin{equation}
\label{eq:invariance U_j}
 \mu\left( U_j^{\eps_j} \vartriangle T_{s_j}U_j^{\eps_j} \right) = O\left(\delta_j+\frac1{\delta_jh_j}\right) = o(s_j).
\end{equation}
Taking a subsequence if necessary, we can assume that the sequence of probability measures $\lambda_j$, defined by
$$
\forall A, B\in\A, \qquad \lambda_j(A\times B)\egdef \frac{1}{\mu(U_j^{\eps_j})} \EE_\mu\left[ \ind{A}\ind{T_{s_j}^{k_j}B}\ind{U_j^{\eps_j}}\right],
$$
converges to some probability measure $\lambda$, which is invariant by $T_{s_p}\times T_{s_p}$ for all $p$ by~\eqref{eq:invariance U_j}.
Recall that $\mu(U_j^{\eps_j})>\mu(Y_j)/2$ and that $\ind{U_j^{\eps_j}}\le\phi_j/\eps_j$.
Then, since $\eps_j>\eps/2$, we have $\lambda|_{\F\otimes\F}\le \frac{4}{\eps} \Delta|_{\F\otimes\F}$.
Since $\Delta|_{\F\otimes\F}$ is an ergodic measure for the flow $\{T_t\times T_t\}|_{\F\otimes\F}$, we can apply Lemma~\ref{lemma:lambda=nu}, which gives $\lambda|_{\F\otimes\F}=\Delta|_{\F\otimes\F}$. 
This means that~\eqref{eq:convergence} remains valid with $\ind{Y_j}$ replaced by $\ind{U_j^{\eps_j}}$. 

The analogue of~\eqref{eq:convergence} is also valid when we replace $\ind{Y_j}$ by $\ind{Y_j\cup U_j^{\eps_j}}$: Indeed, we also have the almost-invariance property
$$\mu\left( (Y_j\cup U_j^{\eps_j}) \vartriangle T_{s_j}(Y_j\cup U_j^{\eps_j}) \right) =o(s_j)$$
and $\ind{Y_j\cup U_j^{\eps_j}}\le \ind{Y_j}+\ind{U_j^{\eps_j}}$. We conclude by a similar argument.

Since $\eps$ can be taken arbitrarily small, we can now use a diagonal argument to show that~\eqref{eq:convergence} remains valid with $\ind{Y_j}$ replaced by $\ind{Y_j\cup U_j^{\eps_j}}$ where the sequence $(\eps_j)$ now satisfies $\eps_j\to0$.
Hence, taking a subsequence if necessary to ensure that $\mu(Y_j\cup U_j^{\eps_j})$ converges to some number $u$, we get 
$$
\forall A,B\in\F , \quad  \EE_\mu\left[ \ind{A}\ind{T_{s_j}^{k_j}B}\ind{Y_j\cup U_j^{\eps_j}}\right] \to u \mu(A\cap B).
$$
It now remains to prove that $u=1$, which we do by repeating the end of the proof of Theorem~\ref{thm:WC}. 
Assume that $u<1$. 
Let us introduce
$$ W_j\egdef \left(\u_{0\le i\le h_j} T_{s_j}^i E_j\right) \setminus Y_j. $$
We have
$$ 
\lim_j \nu(W_j\times W_j) = \lim_j \nu(Y_j^c\times Y_j^c) 
= \lim_j  \EE_\mu \left[ \ind{Y_j^c} (1-\phi_j)\right].
$$
Observe that $(1-\phi_j)\ge\ind{(U_j^{\eps_j})^c} -\eps_j$. Hence
$$
\lim_j \nu(W_j\times W_j) \ge \lim_j  \EE_\mu \left[ \ind{Y_j^c}\ind{(U_j^{\eps_j})^c} \right]
= 1-u >0. 
$$
Let us consider the case where $Y_j$ has the form 
$Y_j^{d,-} =\u_{0\le i< d h_j}T_{s_j}^i E_j$. Then $W_j=\u_{dh_j\le i \le h_j}T_{s_j}^i E_j$, and we define for any $\delta'<1-d$
$$ W_j(\delta')\egdef \u_{(1-\delta')h_j< i \le h_j}T_{s_j}^i E_j\subset W_j. $$
In the same way, if $Y_j$ has the form $Y_j^{d,+} =\u_{(1-d) h_j < i \le h_j}T_{s_j}^i E_j$, we set for $\delta'<1-d$
$$ W_j(\delta')\egdef \u_{0< i <\delta' h_j}T_{s_j}^i E_j\subset W_j. $$
In both cases, note that
$$ \nu\Bigl( (W_j\times W_j) \setminus (W_j(\delta')\times W_j(\delta')) \Bigr) \le 2(1-d-\delta'). $$
thus, for $\delta'$ close enough to $1-d$, we get 
$$ \limsup_j\nu\Bigl(W_j(\delta')\times W_j(\delta')\Bigr)\ge 1-u -2(1-d-\delta') > 0. $$ 
Since $W_j(\delta')\times W_j(\delta')\subset D_j^{\delta'}$, this ensures that 
$$ \limsup \nu(D_j^{\delta'}) > 0. $$
Lemma~\ref{Lemma:approximation} then provides a sequence $(k'_j)$ with $-\delta' h_j\le k'_j \le \delta' h_j$, such that 
$$ \Delta^{k'_j s_j}(\,\cdot\,|C_j^{k'_j}) \tend[w]{j}{\infty} \nu. $$
In particular, $\Delta^{k'_j s_j}(\,\cdot\,|C_j^{k'_j})|_{\F\otimes\F} \tend[w]{j}{\infty} \Delta|_{\F\otimes\F}$. 
Since the projections $Y_j^{k'_j}$ of $C_j^{k'_j}$ on the first coordinate satisfy 
$$ \lim_j \mu (Y_j^{k'_j}) \ge 1-\delta'>d, $$ this contradicts the maximality of $d$. Hence $u=1$.
\end{proof}

\section{King's theorem for flat-roof rank-one flow}

We consider a rank-one flow $\T$. We say that $\T$ has \emph{flat roof} if we can choose the sequence $\xi_j=\{E_j,T_{s_j}E_j,\ldots,T_{s_j}^{h_j-1}E_j, X\setminus\bigsqcup_{k=0}^{h_j-1}T_{s_j}^kE_j\}$ in the definition such that 
$$
\dfrac{\mu\left(T_{s_j}^{h_j}E_j\vartriangle E_j\right)}{\mu(E_j)}\tend{j}{\infty}0.
$$

\begin{theo}
 \label{thm:flat roof}
Let $\T$ be a flat-roof rank-one flow, and $\nu$ be an ergodic self-joining of $\T$. Then there exists a sequence $(k_j)$ such that $\Delta^{k_js_j}\tend[w]{j}{\infty} \nu$.
\end{theo}

\begin{proof}
 Let us defined, for $0\le k\le h_j-1$
$$ a_k^j \egdef \nu\left(T_{s_j}^{k}E_j\times E_j\right)\quad\text{and}\quad b_k^j\egdef\nu\left(E_j\times T_{s_j}^{h_j-k}E_j\right).
$$
We claim that the flat-roof property implies
\begin{equation}
 \label{eq:flat roof property}
h_j\sum_{k=1}^{h_j-1}|a_k^j-b_k^j|\tend{j}{\infty} 0.
\end{equation}
Indeed, by invariance $a_k^j =\nu\left(T_{s_j}^{h_j}E_j\times T_{s_j}^{h_j-k}E_j\right)$. Hence 
$$ |a_k^j-b_k^j| \le \nu\left((T_{s_j}^{h_j}E_j\vartriangle E_j)\times T_{s_j}^{h_j-k}E_j\right), $$
and 
$$ \sum_{k=1}^{h_j-1}|a_k^j-b_k^j| \le \nu\left((T_{s_j}^{h_j}E_j\vartriangle E_j)\times X\right) = \mu\left((T_{s_j}^{h_j}E_j\vartriangle E_j)\right). $$
The claim follows, since $\mu(E_j)\sim 1/h_j$.

\begin{figure}[htp]
 \centering
 \input{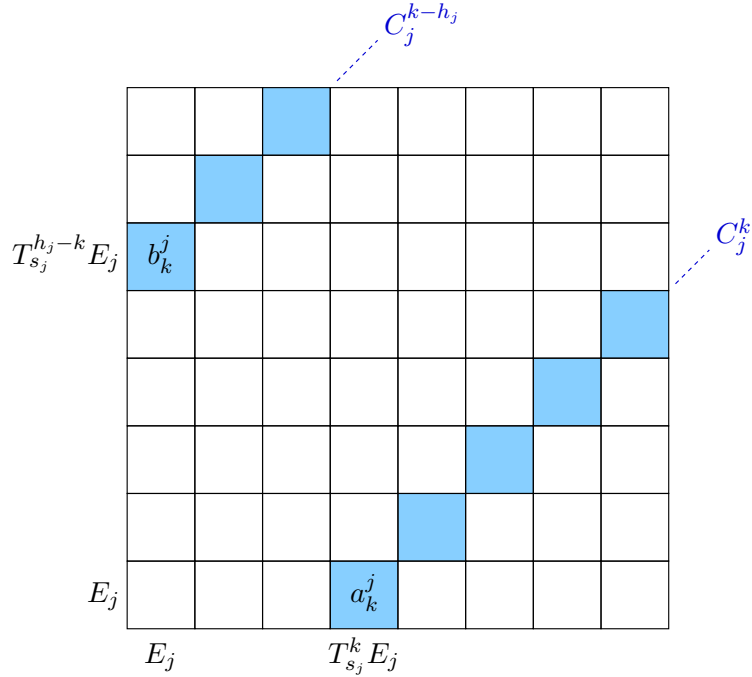}
 \caption{The union of $C_j^k$ and $C_j^{k-h_j}$ is denoted by $G_j^k$.}
 \label{fig:flat_roof}
\end{figure}

We gather the columns $C_j^k$ in pairs, defining for $1\le k\le h_j-1$, $G_j^k\egdef C_j^k\sqcup C_j^{k-h_j}$. (See Figure~\ref{fig:flat_roof}.) We also set $G_j^0\egdef C_j^0$.  Note that 
$\nu(G_j^k) = (h_j-k) a_k^j + k b_k^j$. 
Observe also that 
$$\nu\left( \bigsqcup_{k=0}^{h_j-1} G_j^k \right) =  \nu\left(\bigsqcup_{k=0}^{h_j-1}T_{s_j}^kE_j\times \bigsqcup_{k=0}^{h_j-1}T_{s_j}^kE_j\right)\tend{j}{\infty} 1.$$
Hence,
\begin{equation}
 \label{eq:decomposition in Gjk} 
\sum_{k=0}^{h_j-1} \nu(G_j^k)\,\nu(\,\cdot\,|G_j^k)  \tend[w]{j}{\infty} \nu.
\end{equation}
We claim that, using the flat-roof property, we can in the above equation replace $\nu(\,\cdot\,|G_j^k)$ by $\Delta^{ks_j}$.
Let $A$ and $B$ be $\xi_j$-measurable sets, which are unions of $T_{s_j}^iE_j$ $(0\le i\le h_j-1$). We denote by $r_k$ (respectively $\ell_k$) the number of elementary cells of the form $T_{s_j}^{i_1} E_j\times T_{s_j}^{i_2}E_j$ which are contained in $A\times B$ and which belong to the column $C_j^k$ (respectively $C_j^{k-h_j}$). We have
\begin{equation}
\label{eq:first}
\nu(A\times B|G_j^k) \nu(G_j^k) = \ell_k b_k^j + r_k a_k^j. 
\end{equation}
Moreover, we will show that the flat-roof property ensures the existence of a sequence $(\varepsilon_j)$ with $\varepsilon_j\tend{j}{\infty}0$ such that
\begin{equation}
 \label{eq:approxime Delta_k}
\left| \Delta^{ks_j}(A\times B) - \dfrac{\ell_k+r_k}{h_j}\right| \le \varepsilon_j. 
\end{equation}
Indeed, let us cut $A$ into $A_1\egdef A\cap\bigsqcup_{0\le i\le k-1}T_{s_j}^iE_j$ and $A_2\egdef A\cap\bigsqcup_{k\le i\le h_j-1}T_{s_j}^iE_j$. We have 
$$ \Delta^{ks_j}(A_2\times B) = r_k \mu(E_j), $$
and
$$ \Delta^{ks_j}(A_1\times B) = \ell_k \Delta^{ks_j}(E_j\times T_{s_j}^{h_j-k}E_j) + \Delta^{ks_j}\left((A_1\times B)\setminus C_j^{k-h_j}\right). $$
Recalling that $\Delta^{ks_j}(E_j\times T_{s_j}^{h_j-k}E_j)=\mu(E_j\cap T_{s_j}^{h_j}E_j)$, we get
\begin{equation}
 \label{eq:Delta_k}
\Delta^{ks_j}(A\times B) = (r_k +\ell_k)\mu(E_j)-\ell_k\mu(E_j\setminus T_{s_j}^{h_j}E_j)+ \Delta^{ks_j}\left((A_1\times B)\setminus C_j^{k-h_j}\right).
\end{equation}
The second term of the right-hand side is bounded by $h_j\mu(E_j\Delta T_{s_j}^{h_j}E_j)$, which goes to 0 by the flat-roof property. To treat the last term, we consider the particular case $A=B=\bigsqcup_{0\le i\le h_j-1}T_{s_j}^iE_j$, for which this last term is maximized. We have then 
$$ 1-\Delta^{ks_j}(A\times B) \le 2 \mu\left( X\setminus \bigsqcup_{0\le i\le h_j-1}T_{s_j}^iE_j\right)\tend{j}{\infty}0. $$
On the other hand, \eqref{eq:Delta_k} gives 
$$ \Delta^{ks_j}\left((A_1\times B)\setminus C_j^{k-h_j}\right) = \Delta^{ks_j}(A\times B) - h_j\mu(E_j) + k\mu(E_j\setminus T_{s_j}^{h_j}E_j). $$
Since $h_j\mu(E_j)\to 1$, and $k\mu(E_j\setminus T_{s_j}^{h_j}E_j)\le h_j\mu(E_j\Delta T_{s_j}^{h_j}E_j)\to 0$, we get that the last term of~\eqref{eq:Delta_k} goes to 0 uniformly with respect to $k$, $A$ and $B$. 
It follows that
$$ \left| \Delta^{ks_j}(A\times B) - (\ell_k+r_k)\mu(E_j)\right| \tend{j}{\infty}0, $$
uniformly with respect to $k$, $A$ and $B$. This concludes the proof of~\eqref{eq:approxime Delta_k}.

Equations~\eqref{eq:approxime Delta_k} and~\eqref{eq:first} give
\begin{align*}
 &\sum_{k=0}^{h_j-1} \Bigl|\nu(A\times B|G_j^k) - \Delta^{ks_j}(A\times B))\Bigr| \nu(G_j^k)\\
\le \ & \sum_{k=0}^{h_j-1} |a_k^j-b_k^j| \left| \ell_k - \frac{k}{h_j}(\ell_k + r_k) \right| + \varepsilon_j\\
\le \ & h_j \sum_{k=0}^{h_j-1} |a_k^j-b_k^j| + \varepsilon_j
\end{align*}
which goes to 0 as $j\to\infty$ by~\eqref{eq:flat roof property}. 

Recalling~\eqref{eq:decomposition in Gjk}, we obtain
$$ \sum_{k=0}^{h_j-1} \nu(G_j^k)\Delta^{ks_j} \tend[w]{j}{\infty} \nu. $$
It remains to apply the Choice Lemma to conclude the proof of the theorem.
\end{proof}

\section{$\ZZ^n$-Rank-one action}

We consider now an action of $\ZZ^n$ ($n\ge1$). For $k\in\ZZ^n$, we denote by $k(1),\ldots, k(n)$ its coordinates.

\begin{definition}
\label{def:rank-one Zn}
 A $\ZZ^n$-action $\{T_k\}_{k\in\ZZ^n}$ is of rank one if there exists a sequence $(\xi_j)$ of partitions converging to the partition into points, where $\xi_j$ is of the form
$$\xi_j= \left\{ \left(T_k E_j\right)_{k\in R_j},  X\setminus \u_{k} T_{k} E_j\right\},$$
and $R_j$ is a rectangular set of indices:
$$R_j=\{0,\ldots,h_j(1)-1\}\times\cdots\times\{0,\ldots,h_j(n)-1\}.$$
\end{definition}

Note that the above definition corresponds to so-called \emph{$\mathcal{R}$-rank one actions} defined in~\cite{RS2011} with the additional condition that the shapes in the sequence $\mathcal{R}$ be rectangles.
The sequence $(\xi_j)$ in the above definition being fixed, we define as for the rank-one flows the notions of columns and fat diagonals: For any $k\in\ZZ^n$, we set
$$ 
C^k_j\egdef \u_{\genfrac{}{}{0pt}{}{\scriptstyle r,\ell \in R_j}{\scriptstyle r-\ell = k}} T_{r}E_j\times T_\ell E_j,
$$
and given $0<\delta<1$, 
$$D^\delta_j\egdef\u_{ k:\ \prod_i(h_j(i)-|k(i)|)\ge(1-\delta)\prod_i h_j(i)} C^k_j.$$

\begin{lemma}
\label{lemma:ZnLemma}
 For any self-joining $\nu$ of the rank-one action $\{T_k\}_{k\in\ZZ^n}$, for any $\delta>1-1/2^n$, we have
$$ \liminf_{j\to\infty} \nu(D^\delta_j) > 0. $$
\end{lemma}

\begin{proof}
 We can find $\varepsilon>0$, small enough such that
$$ \left(\frac{1}{2}-\varepsilon\right)^n>1-\delta. $$
Let $r\in\ZZ^n$ be such that
$$ \forall i,\ \left(\frac{1}{2}-\varepsilon\right)h_j(i) < r(i) < \left(\frac{1}{2}+\varepsilon\right)h_j(i). $$
Then, for any $\ell\in R_j$, we have for all $i$: $|r(i)-\ell(i)|<\left(\frac{1}{2}+\varepsilon\right)h_j(i)$. Hence 
$$ \prod_i \Bigl(h_j(i)-|r(i)-\ell(i)|\Bigr) > (1-\delta)\prod_i {h_j(i)}, $$
which means that for any $\ell\in R_j$, the column $C^{r-\ell}_j$ is contained in $D^\delta_j$. 
It follows that
$$ \left(\u_{r:\ \forall i,\ |r(i)-h_j(i)/2|<\varepsilon h_j(i)} T_r E_j\right) \times \left(\u_{\ell\in R_j} T_\ell E_j\right) \subset D^\delta_j. $$
We then get
$$ \liminf_{j\to\infty} \nu(D^\delta_j) \ge \liminf_{j\to\infty}
 \mu\left(\u_{r:\ \forall i,\ |r(i)-h_j(i)/2|<\varepsilon h_j(i)} T_r E_j\right) = (2\varepsilon)^n. $$
\end{proof}

We can now state the analogue of Theorem~\ref{thm:1.2} for $\ZZ^n$-rank-one action, which was first proved by A.A.~Pavlova in~\cite{Pavlova2008}.

\begin{theo}
\label{thm:Pavlova}
 Let $\nu$ be an ergodic self-joining of the $\ZZ^n$-rank-one action $\{T_k\}_{k\in\ZZ^n}$. Then we can find a sequence $(k_j)$ in $\ZZ^n$ and some self-joining $\nu'$ such that 
$\Delta^{k_j} \tend[w]{j}{\infty} \frac{1}{2^n}\nu +\left(1-\frac{1}{2^n}\right)\nu'$: For all measurable sets $A,B$
$$\mu(A\cap T_{k_j}B) \to \frac{1}{2^n}\nu(A\times B) +
\left(1-\frac{1}{2^n}\right)\nu'(A\times B).$$
\end{theo}

\begin{proof}
The proof follows the same lines as for Theorem~\ref{thm:1.2}. 
First note that Lemma~\ref{Lemma:approximation} can be easily adapted to the $\ZZ^n$-situation. Hence, by Lemma~\ref{lemma:ZnLemma}, using a diagonal argument, we get the existence of $(k_j)$ and $(\delta_j)\searrow 1-\frac{1}{2^n}$  with $C_j^{k_j}\subset D_j^{\delta_j}$ such that 
$$\Delta^{k_j}\left( \, \cdot\, | C_j^{k_j}\right) \tend[w]{j}{\infty} \nu. $$
To conclude, it remains to prove that $\liminf\Delta^{k_j}( C_j^{k_j})\ge 1/2^n$. To this aim, we count the number of pairs $(r,\ell)$ such that $T_rE_j\times T_\ell E_j \subset C_j^{k_j}$. We can easily check that these are exactly the pairs $(r,\ell)$ such that, for all $1\le i\le n$, there exists $m(i)\in\{0,\ldots,h_j(i)-1-|k_j(i)|\}$ with
$$ \bigl(r(i),\ell(i)\bigr) = \begin{cases}
                     \bigl(k_j(i)+m(i),m(i)\bigr) & \text{if $k_j(i)\ge0$ }\\
		     \bigl(m(i), -k_j(i)+m(i)\bigr)  & \text{otherwise}.
                    \end{cases}
$$ 
Hence $\Delta^{k_j}( C_j^{k_j})=\prod_i \Bigl(h_j(i)-1-|k_j(i)|\Bigr)\mu(E_j)$. Using the fact that $C_j^{k_j}\subset D_j^{\delta_j}$, we get the desired result.
\end{proof}

When $n\ge2$, it is known that the Weak Closure Theorem fails (counterexamples have been given in~\cite{DK2002,DS2009}). However, as a consequence of Theorem~\ref{thm:Pavlova}, we get the following:

\begin{corollary}[Partial Weak Closure Theorem for $\ZZ^n$-rank-one action]
\label{Cor:partialWCT}
Let $S$ be an automorphism commuting with the $\ZZ^n$-rank-one action $\{T_k\}_{k\in\ZZ^n}$. 
Then we can find a sequence $(k_j)$ in $\ZZ^n$ and some self-joining $\nu'$ such that 
$$ \Delta^{k_j}\tend{j}{\infty}\dfrac{1}{2^n}\Delta_S+\left(1-\frac{1}{2^n}\right)\nu'. $$
Moreover, if $S\notin\{T_k\,\ k\in\ZZ^n \}$, then $\{T_k\}_{k\in\ZZ^n}$ is partially rigid: There exists a sequence $(k'_\ell)$ in $\ZZ^n$ with $|k'_\ell|\to\infty$ such that for all measurable sets $A$ and $B$
$$ \liminf_{\ell\to\infty} \mu\left(A\cap T_{k'_\ell}B\right)\ge \dfrac{1}{2^{2n}}\mu(A\cap B). $$
\end{corollary}

\begin{proof}
The first part is a direct application of Theorem~\ref{thm:Pavlova} with $\nu=\Delta_S$. If moreover $S\notin\{T_k\,\ k\in\ZZ^n \}$, then the sequence $(k_j)$ of the theorem must satisfy $|k_j|\to \infty$. Let us enumerate the cylinder sets as $\{A_0,A_1,\ldots,A_\ell,\ldots\}$.
Let $(\eps_\ell)$ be a sequence of positive numbers decreasing to zero. For any $\ell$, we can find a large enough integer $j_1(\ell)$ such that for all cylinder sets $A,B\in\{A_0,A_1,\ldots,A_\ell\}$, 
$$ \mu\left(T_{k_{j_1(\ell)}}A\cap SB\right) \ge \left(\dfrac{1}{2^n} - \eps_\ell \right) \mu (SA\cap SB)
=\left(\dfrac{1}{2^n} - \eps_\ell \right) \mu (A\cap B). $$
Then, we can find  a large enough integer $j_2(\ell)$ with $|j_2(\ell)|>2|j_1(\ell)|$ such that for all cylinder sets $A,B\in\{A_0,A_1,\ldots,A_\ell\}$, 
$$ \mu\left(T_{k_{j_1(\ell)}}A\cap T_{k_{j_2(\ell)}}B\right) \ge \left(\dfrac{1}{2^n} - \eps_\ell \right) \mu (T_{k_{j_1(\ell)}}A\cap SB). $$ 
It follows that for all $\ell\ge0$ and all cylinder sets $A,B\in\{A_0,A_1,\ldots,A_\ell\}$,
$$ \mu\left(A\cap T_{k_{j_2(\ell)}-k_{j_1(\ell)}}B\right) \ge \left(\dfrac{1}{2^n} - \eps_\ell \right)^2 \mu (A\cap B).$$
This proves the result announced in the corollary when $A$ and $B$ are cylinder sets with $k'_\ell\egdef k_{j_2(\ell)}-k_{j_1(\ell)}$, and this extends in a standard way to all measurable sets.
\end{proof}

The counterexample given in~\cite{DS2009} also shows that the rigidity of factors is no more valid when $n\ge2$. Theorem~\ref{thm:Pavlova} only ensures the partial rigidity of factors of $\ZZ^n$-rank-one actions.

\begin{corollary}[Partial rigidity of factors of $\ZZ^n$-rank-one action]
\label{Cor:partial rigidity}
  Let $\F$ be a non-trivial factor of the $\ZZ^n$-rank-one action $\{T_k\}_{k\in\ZZ^n}$. 
Then there exists a sequence $(k_j)$ in $\ZZ^n$ with $|k_j|\to \infty$ such that, for all measurable sets $A,B\in\F$
$$ \liminf\mu(A\cap T_{k_j}B ) \ge \dfrac{1}{2^{n}}\mu(A\cap B).$$
\end{corollary}

\begin{proof}
 This is a direct application of Theorem~\ref{thm:Pavlova} where $\nu$ is an ergodic component of the relatively independent joining above the factor $\F$. 
\end{proof}

\bibliography{rank-one}

\end{document}